\documentclass[11pt,a4paper]{article}
\title{Every $d(d+1)$-connected graph is globally rigid in $\R^d$}
\author{Soma Villányi\thanks{HUN-REN-ELTE Egerv\'ary Research Group
on Combinatorial Optimization,
P\'azm\'any P\'eter s\'et\'any 1/C, 1117 Budapest, Hungary.
e-mail: {\tt villanyi.soma@gmail.com}}}

\usepackage[utf8]{inputenc}
\usepackage[english]{babel}
\usepackage{amsmath}
\usepackage{amsthm}
\usepackage{amssymb}
\usepackage{geometry}
\usepackage{a4wide}
\usepackage{bm}
\usepackage{microtype}
\usepackage{mathtools}
\usepackage{csquotes}
\usepackage[shortlabels]{enumitem}
\usepackage{algorithm}
\usepackage{algpseudocode}
\usepackage[nottoc,numbib]{tocbibind}
\usepackage{url}
\usepackage{changepage}
\usepackage{enumitem}
\usepackage{setspace}
\usepackage{graphicx,accents}

\theoremstyle{definition}
\newtheorem{theorem}{Theorem}[section]
\newtheorem{lemma}[theorem]{Lemma}
\newtheorem{claim}[theorem]{Claim}

\makeatletter
\newcommand{\overleftsmallarrow}{\mathpalette{\overarrowsmall@\leftarrowfill@}}
\newcommand{\overarrowsmall@}[3]{%
  \vbox{%
    \ialign{%
      ##\crcr
      #1{\smaller@style{#2}}\crcr
      \noalign{\nointerlineskip}%
      $\m@th\hfil#2#3\hfil$\crcr
    }%
  }%
}
\def\smaller@style#1{%
  \ifx#1\displaystyle\scriptstyle\else
    \ifx#1\textstyle\scriptstyle\else
      \scriptscriptstyle
    \fi
  \fi
}
\makeatother
\makeatletter
\DeclareRobustCommand{\cev}[1]{%
  \mathpalette\do@cev{#1}%
}

\newcommand{\N}{{ N}}

\newcommand{\piny}{\overleftsmallarrow{\hspace{2pt} \pi \hspace{1pt}}}
\newcommand{\R}{\mathbb{R}}

\newcommand{\Mat}{{\mathcal{M}}}

\newcommand{\E}{\mathbb{E}}
\newcommand{\Prob}{\mathbb{P}}

\setlist[enumerate]{itemsep=0mm,parsep=2mm}
\begin{document}
\allowdisplaybreaks 
\maketitle
\vspace{-0.3cm}
\begin{abstract}
Using a probabilistic method, we prove that $d(d+1)$-connected graphs are rigid in $\R^d$, a conjecture of Lov\'asz and Yemini \cite{LY}.
Then, using recent results on weakly globally linked pairs, we modify our argument to prove that $d(d+1)$-connected graphs are globally rigid, %
 too, a conjecture of Connelly, Jord\'an and Whiteley \cite{CJW}.
  The constant $d(d+1)$ is best possible.
\end{abstract}

\section{Introduction}
A well-known theorem of Lovász and Yemini dating back to 1982 states that every 6-connected graph is rigid\footnote{ For definitions see the next section.} in $\R^2$ \cite{LY}. In 2005, Jackson and Jordán   proved that 6-connected graphs are globally rigid, too, in $\R^2$ \cite{JJconnrig}. 
However, for $d\geq 3$, it has been a long-standing open problem whether there exists a constant $c_d$ such that every $c_d$-connected graph is rigid (or globally rigid) in $\R^d$. The rigid version was conjectured by Lovász and Yemini, who \mbox{also stated} that (for $d\geq 2$) the optimal value of $c_d$ could be $d(d+1)$, and constructed  an infinite family of $(d(d+1)-1)$-connected graphs that are not rigid in $\R^d$. 
 The stronger version of this conjecture, which concerns global rigidity,
 was proposed by Connelly, Jordán and \mbox{Whiteley \cite{CJW}}. 
 
 In this paper, we verify these statements. 

\begin{theorem}\label{ddplus1 rigid}
Every $d(d+1)$-connected graph is rigid in $\R^d$.
\end{theorem}

\begin{theorem}\label{ddplus1 globrigid}
Every $d(d+1)$-connected graph is globally rigid in $\R^d$.
\end{theorem}

The following partial results were known about these problems. It follows from a theorem of Tanigawa that if every $k$-connected graph is rigid in $\R^d$, then $(k+1)$-connected graphs are globally rigid in $\R^d$. Hence, regarding the existence of some suitable $c_d$, the rigid version of the conjecture was known to be equivalent to the globally rigid version \cite{Tan}.
In some important graph families sufficiently high connectivity  was known to imply rigidity or global rigidity in $\R^3$, e.g., among molecular graphs \cite{Jmol} and body-hinge graphs \cite{JKT}.
 Krivelevich, Lew and Michaeli  proved  recently
that there exists $C > 0$ such that every $(Cd^2
\log^2n)$-connected graph on $n$ vertices is rigid in $\R^d$ \cite{KLM}. 
It was also proved recently, by Clinch, Jackson and Tanigawa,  that if a graph is 12-connected, then it is generically $C_2^1$-rigid, where $C_2^1$ denotes a so-called cofactor matroid, which is conjectured to equal the 3-dimensional rigidity \mbox{matroid \cite{CJT2}}. We will prove a conjecture of Whiteley, which generalizes this result; for more details see the end of the paper.

Theorem \ref{ddplus1 globrigid} is a stronger statement than Theorem \ref{ddplus1 rigid}. Nevertheless, we will 
 first consider Theorem \ref{ddplus1 rigid} separately because its proof only uses  classical results of rigidity theory. On the other hand, in the proof of Theorem \ref{ddplus1 globrigid}, we will rely on some very recent results on {weak global linkedness}, a notion which will be introduced in the next section. 
 
The rest of the paper is organised as follows. In Section \ref{sec:pre}, we introduce the necessary notions and prove a combinatorial lemma. In Section \ref{sec:proof1}, we prove Theorem \ref{ddplus1 rigid}. In Section \ref{sec:proof2}, we slightly alter this argument and prove Theorem \ref{ddplus1 globrigid}. In Section \ref{sec:sharp}, we extend our results by proving that $d(d+1)$-connected graphs are $\binom{d+1}{2}+1$-redundantly rigid and $\binom{d+1}{2}$-redundantly globally rigid in $\R^d$, where the constants $\binom{d+1}{2}+1$ and $\binom{d+1}{2}$ are best possible. The paper ends with some concluding remarks.

\section{Preliminaries}\label{sec:pre}
\subsection{Rigidity and global rigidity}
We briefly introduce the definitions and the results that we need. For a proper introduction to graph rigidity, the reader is directed to, e.g., \cite{GaG} or \cite{WhitM}.
For recent surveys, see \cite{JW,SW}.

A $d$-dimensional {\it framework}
is a pair $(G,p)$, where $G=(V,E)$
is a graph and $p$ is a map from $V$ to $\mathbb{R}^d$. 
Two frameworks
$(G,p)$ and $(G,q)$ are {\it equivalent} if corresponding edge lengths are the same, that is, 
$||p(u)-p(v)||=||q(u)-q(v)||$ holds
for all pairs $u,v$ with $uv\in E$.
The frameworks $(G,p)$ and $(G,q)$ are {\it congruent} if
$||p(u)-p(v)||=||q(u)-q(v)||$ holds
for all pairs $u,v$
with $u,v\in V$.
A $d$-dimensional framework $(G,p)$ is called {\it rigid} if there exists some $\varepsilon
>0$ such that if $(G,q)$ is equivalent to $(G,p)$ and
$||p(v)-q(v)||< \varepsilon$ for all $v\in V$, then $(G,q)$ is
congruent to $(G,p)$. 
The framework $(G,p)$ is called {\it globally rigid} if every equivalent 
$d$-dimensional framework
$(G,q)$ is congruent to $(G,p)$. 

A framework $(G,p)$ is said to be {\it generic} if the set of
its $d|V(G)|$ vertex coordinates is algebraically independent over $\mathbb{Q}$.
It is known that in a given dimension 
the (global) rigidity of a generic framework $(G,p)$ depends only on $G$: either every generic framework of $G$ in $\R^d$ is (globally) rigid, or none of them are  (\cite{AR} , \cite{Con, GHT}).
Thus, we say that a graph $G$ is {\it (globally) rigid} in $\R^d$ if
every (or equivalently, if some) 
generic $d$-dimensional framework of $G$ is (globally) rigid in $\R^d$. It is well-known that if a graph is rigid (resp. globally rigid) in $\R^d$, then it is $d$-connected (resp. $(d+1)$-connected).  If $d=1$, then these conditions are also sufficient for rigidity and global rigidity, respectively. If two complete graphs share at least $d$ vertices, then their union is rigid in $\R^d$.
A graph $G$ is said to be {\it $t$-redundantly (globally) rigid} in $\R^d$ if for every subset $E_0$ of $E(G)$ with $|E_0|< t$, the graph $G-E_0$ is (globally) rigid in $\R^d$.

The \emph{rigidity matrix} of the framework $(G,p)$
is the matrix $R(G,p)$ of size
$|E|\times d|V|$, where, for each edge $uv\in E$, in the row
corresponding to $uv$,
the entries in the $d$ columns corresponding to vertices $u$ and $v$ contain
the $d$ coordinates of
$(p(u)-p(v))$ and $(p(v)-p(u))$, respectively,
and the remaining entries
are zeros. 
The rigidity matrix of $(G,p)$ defines
the \emph{rigidity matroid}  of $(G,p)$ on the ground set $E$
by linear independence of the rows. %
It is known that any pair of generic frameworks
$(G,p)$ and $(G,q)$ have the same rigidity matroid.
We call this the $d$-dimensional \emph{rigidity matroid}
${\cal R}_d(G)=(E,r_d)$ of the graph $G$.

We denote the rank of ${\cal R}_d(G)$ by $r_d(G)$.
A graph $G=(V,E)$ is called \emph{${\cal R}_d$-independent} if 
 $r_d(G)=|E|$.
The pair $\{u,v\}$ is called {\it linked} in $G$ in $\R^d$ if $uv\in E$ or
$r_d(G+uv)=r_d(G)$ holds. If $\{u,v\}$ is not linked, then it is called {\it loose} in $G$. Thus, in an $\mathcal{R}_d$-independent graph $G$ a non-adjacent pair $\{u,v\}$ is loose if and only if $G+uv$ is $\mathcal{R}_d$-independent. Every graph on at most $d+1$ vertices is $\mathcal{R}_d$-independent.

The so-called ($d$-dimensional) {\it Henneberg operations} are two graph operations, which are fundamental in this context.  The {\it 0-extension} operation on vertices $v_1,\dots,v_d$ adds a new vertex $v$ and new edges $vv_1,\dots,vv_d$ to the graph. The {\it 1-extension} operation on an edge $v_1v_2$ and vertices $v_3,\dots,v_{d+1}$ deletes the edge $v_1v_2$ and adds a new vertex $v$ and new edges $vv_1,\dots, vv_{d+1}$ to the graph. These operations are known to preserve $\mathcal{R}_d$-independence.%

The following theorem %
is due to Gluck \cite{Gluck}. 
\begin{theorem}\cite{Gluck}\label{theorem:gluck}
\label{combrigid}
Let $G=(V,E)$ be a graph with $|V|\geq d+1$.
 Then $r_d(G)\leq d|V|-\binom{d+1}{2}$. Furthermore, $G$ is rigid in $\R^d$
if and only if $r_d(G)=d|V|-\binom{d+1}{2}$.
\end{theorem}
The first part of Theorem \ref{theorem:gluck} implies the following statement.
\begin{lemma}\label{rankbound}
Let $G=(V,E)$ be a graph. Suppose that $E=E_0\cup E_1\cup \dots \cup E_s$ with $|V(E_i)|\geq d+1$ for $1\leq i\leq s$. Then $r_d(G)\leq |E_0|+\sum_{i=1}^s \left( d|V(E_i)|-\binom{d+1}{2}\right)$.
\end{lemma}

At this point let us recall from \cite{LY} how $k$-connected non-rigid graphs can be constructed in $\R^d$ when $k=d(d+1)-1$ and $d\geq 2$. Let $G_0=(V_0,E_0)$ be a $k$-regular $k$-connected graph on $s$ vertices. Split every vertex
of $G_0$ into  $k$ vertices of degree 1 and then, for each $v\in V_0$, add a complete graph $G_v$ on the $k$ vertices that represent $v$. An example with $k=5$, $d=2$ and $s=8$ is shown in Figure \ref{fig:nonstiff}. It is easy to see that the resulting graph $G=(V,E)$ is $k$-connected.
  Furthermore, $E=E_0\cup \bigcup_{v\in V_0}E(G_v)$ and hence by Lemma \ref{rankbound}

 \small
  \begin{align*}
  r_d(G)\leq
   \frac{ks}{2} + s\left(dk-\binom{d+1}{2}\right)=
  \left(d+\frac{1}{2}-\frac{1}{2}\cdot\frac{d(d+1)}{k}\right)ks=
  \left(d+\frac{1}{2}-\frac{1}{2}\cdot\frac{d(d+1)}{k}\right)|V(G)|.
  \end{align*}
  \normalsize

  If $s$ is large enough, then the right-hand side is less than $d|V(G)|-\binom{d+1}{2}$ and hence $G$ is not rigid in $\R^d$ by Theorem \ref{theorem:gluck}.

 \begin{figure}[b!]
\begin{center}
\includegraphics[scale=0.5]{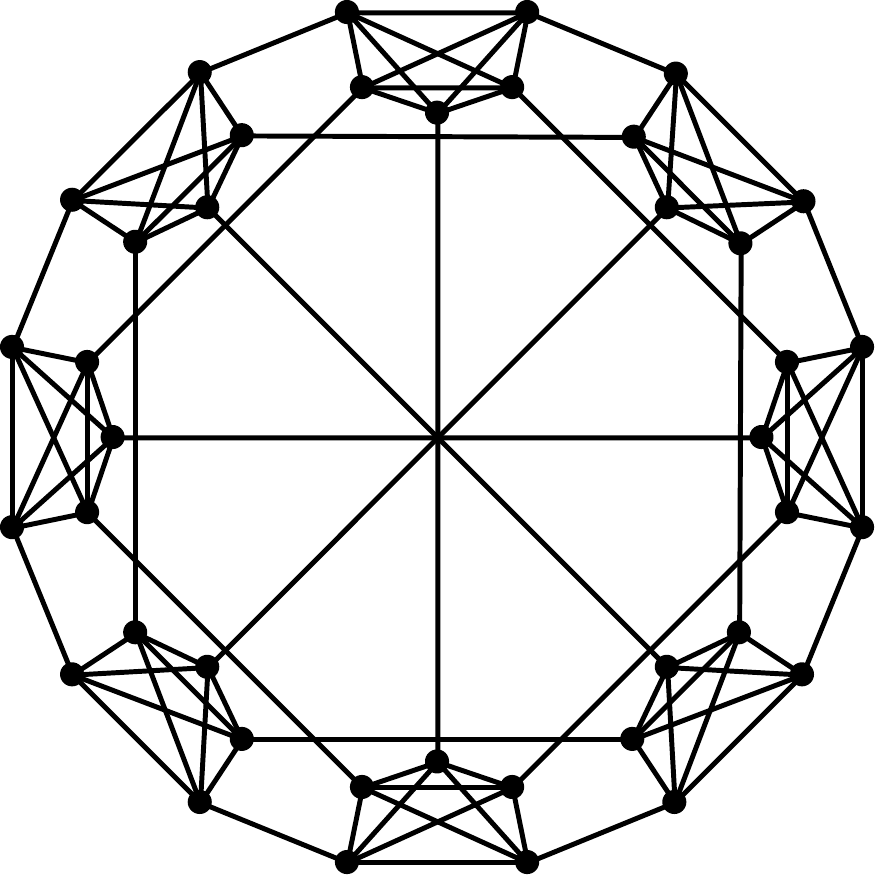}
\end{center}
\caption{A 5-connected graph which is not rigid in $\R^2$.}
\label{fig:nonstiff}
\end{figure}

\newpage
\subsection{Weakly globally linked pairs}
We will need the following notion in the proof of Theorem \ref{ddplus1 globrigid}. 

A pair $\{u,v\}$ of vertices  in a framework $(G,p)$ is called {\it globally linked in $(G,p)$} if for every equivalent framework $(G,q)$ we have
$||p(u)-p(v)||=||q(u)-q(v)||$. Global linkedness in $\R^d$ is not a generic property (for $d\geq 2$): 
a vertex pair may be globally linked in some generic $d$-dimensional frameworks of $G$ without being globally linked in all generic frameworks. 
We say that a pair $\{u,v\}$ is {\it weakly globally linked in $G$} in $\R^d$ if there exists a generic $d$-dimensional framework $(G,p)$
in which $\{u,v\}$ is globally linked. If  $\{u,v\}$ is not weakly globally linked in $G$ in $\R^d$, then it is called {\it globally loose} in $G$ in $\R^d$. If a vertex pair is loose in $G$, then it is also globally loose.
The following lemmas were proved in \cite{wgl}.
\begin{lemma}\cite{wgl}\label{J}
Let $G=(V,E)$ be a graph. Suppose that $\{u,v\}$ is a non-adjacent weakly globally linked vertex pair in $G$ in $\R^d$. Then $G$ is globally rigid in $\R^d$ if and only if $G+uv$ is globally rigid in $\R^d$.
\end{lemma}

\begin{lemma}\cite{wgl}\label{wglpath} %
Let $G=(V,E)$ be a graph. Suppose that $V_0\subset V$ with $u,v\in V_0$, $\{u,v\}$ is linked in $G[V_0]$ in $\R^d$ and there exists a $u$-$v$-path in $G$ that is internally disjoint from $V_0$. Then $\{u,v\}$ is weakly globally linked in $G$ in $\R^d$.
\end{lemma}

\subsection{A combinatorial lemma}
For $n\in \mathbb{N}$, the set $\{1,\dots,n\}$ is denoted by $[n]$. For a set $H$, $\binom{H}{n}=\{S\subseteq H: |S|=n\}$.

\begin{lemma}\label{comblemma}
Let $n,r,d,m\in \mathbb{N}$ with $2\leq d$ and  $d+1\leq m\leq n-1$. 
Suppose that $H_1,\dots, H_r\subsetneq [n]$ such that %
 $H_{j_1}\neq H_{j_2}$ and  $|H_{j_1}\cap H_{j_2}|\leq d-2$ for every $1\leq j_1<j_2\leq r$. Then $$\left|\left\{S\in \binom{[n]}{m} :\exists j\in \{1,\dots,r\}, S\subseteq H_j\right\}\right|\leq \binom{n-1}{m}. $$
\begin{proof}
For $l\in \{m,m-1\}$, let  $\mathcal{S}_l=\left\{S\in \binom{[n]}{l} :\exists j\in \{1,\dots,r\}, S\subseteq H_j\right\}$. Let $|\mathcal{S}_l|$ be denoted by $A_l$. Suppose, for a contradiction, that $A_{m}>\binom{n-1}{m}$.  A counting argument shows that

 $$A_{m-1}((n-1)-(m-1))\geq \big|\big\{(S,S')\in \mathcal{S}_{m-1}\times S_m: S\subseteq S'\big\}\big| = A_{m}m.$$

\noindent It follows that  $$A_{m-1}\geq A_{m}\frac{m}{n-m}> \binom{n-1}{m}\frac{m}{n-m}=\binom{n-1}{m-1}.$$

\noindent For each $S\in \mathcal{S}_{m-1}$ %
let $j_S$ denote the index for which $S\subseteq H_{j_S}$ and choose an arbitrary element $u_S\in [n]-H_{j_S}$. Since $m-1> d-2$, there is no $j$ with $S\cup \{u_S\}\subseteq H_{j}$. Hence, 
$$f:S\mapsto S\cup \{u_S\} \text{ is a function from } \mathcal{S}_{m-1}  \text{ into } \binom{[n]}{m}-\mathcal{S}_m.$$ 
Suppose that $S,S'\in \mathcal{S}_{m-1}$ and $S'\cup \{u_{S'}\}=S\cup \{u_S\}$. Then $|S'\cap S|\geq m-2> d-2$. Therefore, we must have $H_{j_{S'}}=H_{j_{S}}$ and also $S'=S$.
It follows that $f$ is an injection and $A_{m-1}\leq \binom{n}{m}-A_{m}.$ Hence,

 $$A_{m}\leq \binom{n}{m}-A_{m-1} < \binom{n}{m}-\binom{n-1}{m-1} = \binom{n-1}{m} ,$$ a contradiction.
\end{proof}
   
\end{lemma}

\section{Proof of Theorem \ref{ddplus1 rigid}}\label{sec:proof1}
Fix $d\geq 2$. Let $G=(V,E)$ be a graph and $\pi =(v_1,\dots,v_n)$ an ordering of the vertices of $G$. Let $\N_{G,\piny}(v_i)=\{u\in \N_G(v_i)$: $u$ precedes $v_i$ in $\pi\}$  %
and let $\deg_{G,\piny}(v_i)$ denote $|\N_{G,\piny}(v_i)|$.%

 Construct  a subgraph  $G_\pi=(V,E_\pi)$ of $G$ according to the following rules. %
\begin{enumerate}[(a)]
\item If $\deg_{G,\piny}(v_i)\leq d$, then in $G_\pi$ connect $v_i$ with every vertex of $\N_{G,\piny}(v_i)$.
\item If $\deg_{G,\piny}(v_i)\geq d+1$ and $\N_{G,\piny}(v_i)$ induces a clique in $G$, then in $G_\pi$ connect $v_i$ with $d$ vertices of $\N_{G,\piny}(v_i)$.
\item If $\deg_{G,\piny}(v_i)\geq d+1$ and $\N_{G,\piny}(v_i)$ does not induce a clique in $G$, then in $G_\pi$  connect $v_i$ with $d+1$ vertices of $\N_{G,\piny}(v_i)$, including two vertices $x$ and $y$ that are not adjacent in $G$.
\end{enumerate}

\begin{lemma}\label{closedlemma}
Let $G=(V,E)$ be a graph and $\pi =(v_1,\dots,v_n)$ an ordering of the vertices of $G$. %
Suppose that every linked pair is connected in $G$. Then $G_\pi$ is $\mathcal{R}_d$-independent.
\end{lemma}
\begin{proof}
We prove that $F_i=G_\pi[\{v_1,\dots, v_i\}]$ is $\mathcal{R}_d$-independent by induction on $i$. For $i\leq d$, we have $|V(F_i)|\leq d$, which implies that $F_i$ is $\mathcal{R}_d$-independent. Suppose that $d+1\leq i\leq n$.
If (a) or (b) holds, then $F_i$
is a subgraph of a 0-extension of $F_{i-1}$, and thus $F_{i}$ is $\mathcal{R}_d$-independent.
Suppose that (c) holds. %
As $xy\notin E$, $\{x,y\}$ is loose in $G$ and hence in $F_{i-1}$. It follows that $F_{i-1}+xy$ is $\mathcal{R}_d$-independent. $F_{i}$ is a 1-extension of $F_{i-1}+xy$, and thus $F_{i}$ is $\mathcal{R}_d$-independent. 

We have $F_n=G_\pi$, which completes the proof.
\end{proof}

\begin{lemma}\label{Epilemma}
Let $G=(V,E)$ be a graph. Suppose that for each $v\in V$, $\deg_G(v)\geq d(d+1)$, $N_G(v)$ does not induce a clique in $G$, and if $H_1$ and $H_2$ are the vertex sets of two different maximal cliques of $G[N_G(v)]$, then $|H_1\cap H_2|\leq d-2$.  Let $\pi$ be a uniformly random ordering of $V$.
 Then $$\E(|E_\pi|)\geq d|V|.$$
\end{lemma}
\begin{proof}
Fix $v\in V$ and let $k$ denote $\deg_G(v)$. For $0\leq i\leq k$, we have $\Prob\big(\deg_{G,\piny}(v)=i\big)\;=\;\frac{1}{k+1}$. Hence,
\begingroup
\addtolength{\jot}{0.2em}
\begin{align}
\begin{split}
\E(\min(\deg_{G,\piny}(v),d))&=\sum_{i=0}^{k}\frac{1}{k+1}\min(i,d)\hspace{4cm}\\
&=\frac{1}{k+1}\left(dk- \frac{d(d-1)}{2}\right)\\
&= d-\frac{1}{2}\cdot\frac{d(d+1)}{k+1}. \\
\end{split}
\end{align}
\endgroup

 \noindent Let $H_1,\dots, H_r$ denote the vertex sets of the maximal cliques of $G[N_G(v)]$.
For $d+1\leq i\leq k$, let
\small
 $$\mathcal{S}_i=\left\{S\in \binom{N_G(v)}{i}: S \text{ induces a clique in } G\right\} =\left\{S\in \binom{N_G(v)}{i}:\exists j\in\{1,\dots,r\}, S\subseteq H_j\right\}.$$
 \normalsize
 Let $A_i$ denote $|\mathcal{S}_i|.$
 Then $A_{k}=0$ and, by Lemma \ref{comblemma}, $A_i\leq \binom{k-1}{i}$ for $d+1\leq i\leq k-1$. 
  Let $Q$ denote the event that $\deg_{G,\piny}(v)\geq d+1$ and $\N_{G,\piny}(v)$ does not induce a clique in $G$. If $\deg_{G,\piny}(v)=i\geq d+1$, then  $Q$ occurs if and only if $\N_{G,\piny}(v)\notin \mathcal{S}_{i}$.
Hence,

\begingroup
\addtolength{\jot}{0.2em}
\begin{align*}
 \Prob(Q|\deg_{G,\piny}(v)=i)&=1-\frac{A_i}{\binom{k}{i}} \geq 1-\frac{\binom{k-1}{i}}{{\binom{k}{i}}}=1- \frac{k-i}{k}=\frac{i}{k}.\hspace{2cm}
 \end{align*}
It follows that
\begin{align}
\begin{split}
\Prob(Q)&=\sum_{i=d+1}^{k} \Prob(\deg_{G,\piny}(v)=i)\Prob(Q|\deg_{G,\piny}(v)=i)\hspace{-0.3cm}\\
&\geq \sum_{i=d+1}^{k} \frac{1}{k+1}\cdot\frac{i}{k}\\
&=\frac{1}{2}\cdot\frac{(k+d+1)(k-d)}{k(k+1)}\\
&=\frac{1}{2} - \frac{1}{2}\cdot \frac{d(d+1)}{k(k+1)}.
\end{split}
\end{align}

\noindent If $Q$ does not occur, then $\deg_{G_\pi,\piny}(v) = \min(\deg_{G,\piny}(v),d)$.\newline If $Q$ occurs, then $\deg_{G_\pi,\piny}(v)= d+1=\min(\deg_{G,\piny}(v),d)+1$.\newline Hence, by combining (1) and (2) we obtain
\begin{align}
\begin{split}
\E(\deg_{G_\pi,\piny}(v))&=\E(\min(\deg_{G,\piny}(v),d))+\Prob(Q)\hspace{3.5cm}\\
&\geq d +\frac{1}{2} -\frac{1}{2}\cdot \frac{d(d+1)}{k}\\
&\geq d.
\end{split}
\end{align}
\endgroup

\noindent Thus
\begin{equation}
\hspace{-0.4cm}E(|G_\pi|) = \E\left(\sum_{v\in V} \deg_{G_\pi,\piny}(v)\right) = \sum_{v\in V} \E(\deg_{G_\pi,\piny}(v))\geq d|V|.\tag*{\qedhere}
\end{equation}
\end{proof}

\begin{proof}[Proof of Theorem \ref{ddplus1 rigid}]
 Suppose, for a contradiction, that $G=(V,E)$ is a $d(d+1)$-connected graph that is not rigid in $\R^d$. We may
assume that $G$ has the least possible number $|V|$ of vertices among all such graphs. We
may also assume that $G$ has the largest number of edges among all such graphs on $|V|$ vertices. 
Then, for each $v\in V$, $\N_G(v)$  does not induce a clique in $G$, for otherwise deleting $v$ would result in a smaller counterexample. (Deleting a vertex whose neighbour set induces a clique preserves $k$-connectivity, unless the graph is a complete graph on $k+1$-vertices.)  Furthermore, every linked pair is connected in $G$, for otherwise connecting a non-adjacent linked pair would result in a counterexample with more edges. In particular, the rigid induced subgraphs of $G$ are complete. 

Suppose that $v\in V$ and $H_1, H_2$ are the vertex sets of two different maximal cliques of $G[\N_G(v)]$. Then $G[H_1\cup H_2\cup \{v\}]$ is non-complete and hence non-rigid. It follows that $|(H_1\cup \{v\}) \cap (H_2\cup \{v\})|\leq d-1$. Thus $|H_1\cap H_2|\leq d-2$.

 Hence, by Lemma \ref{Epilemma}, if $\pi$ is a uniformly random ordering of $V$, then   $\E(|E_\pi|)\geq d|V|.$
 This is a contradiction %
by Theorem \ref{theorem:gluck} and the fact that $G_\pi$ is $\mathcal{R}_d$-independent by {Lemma \ref{closedlemma}}. \end{proof} %

\section{Proof of Theorem \ref{ddplus1 globrigid}}\label{sec:proof2}

Fix $d\geq 2$. We adapt the argument of the previous section to prove Theorem \ref{ddplus1 globrigid}. We will need the following strengthening of Lemma \ref{closedlemma}. 

\begin{lemma}\label{tlooselemma}
Let $G=(V,E)$ be a graph and $\pi =(v_1,\dots,v_n)$ an ordering of the vertices of $G$. %
Suppose that every weakly globally linked pair is connected in $G$. Then $G_\pi$ is $\mathcal{R}_d$-independent.
\end{lemma}
\begin{proof}
We prove that $F_i=G_\pi[\{v_1,\dots, v_i\}]$ is $\mathcal{R}_d$-independent by induction on $i$. For $i\leq d$, we have $|V(F_i)|\leq d$, which implies that $F_i$ is $\mathcal{R}_d$-independent. Suppose that $d+1\leq i\leq n$.
If (a) or (b) holds, then $F_i$
is a subgraph of a 0-extension of $F_{i-1}$, and thus $F_{i}$ is $\mathcal{R}_d$-independent.
Suppose that (c) holds. As $xy\notin E$, $\{x,y\}$ is globally loose in $G$. Since in $G$ there is an $x$-$y$-path of length two through $v$, the pair $\{x,y\}$ is loose in $F_{i-1}$ by Lemma \ref{wglpath}. It follows that $F_{i-1}+xy$ is $\mathcal{R}_d$-independent. $F_{i}$ is a 1-extension of $F_{i-1}+xy$, and thus $F_{i}$ is $\mathcal{R}_d$-independent. 

We have $F_n=G_\pi$, which completes the proof.
\end{proof}

\begin{proof}[Proof of Theorem \ref{ddplus1 globrigid}]
Suppose, for a contradiction, that $G=(V,E)$ is a $d(d+1)$-connected graph that is not globally rigid in $\R^d$. Again, we may
assume that $G$ is a counterexample with a minimum number of vertices and among these graphs $G$ has a maximum number of edges. 
Then $\N_G(v)$  does not induce a clique, for each $v\in V$. 
Furthermore, 
 every weakly globally linked pair is connected in $G$, for otherwise, by Lemma \ref{J}, connecting a non-adjacent weakly globally linked pair would result in a counterexample with more edges.%
\vspace{-15pt}
\begin{adjustwidth}{15pt}{0pt}
\begin{claim}
Let $v\in V$ and let $H_1$ and $H_2$ be the vertex sets of two different maximal cliques of $G[{ N}_G(v)]$. Then $|H_1\cap H_2|\leq d-2$.\vspace{-0.2cm}
\end{claim}
\begin{proof}
Suppose, for a contradiction, that $|H_1\cap H_2|\geq d-1$. 

We claim that $G[H_1\cup H_2]$ is rigid. If $|H_1\cap H_2|\geq d$, this is obvious. Suppose that $|H_1\cap H_2|= d-1$. Since $G$ is $d+1$-connected, there is a path $P$ in $G-((H_1\cap H_2)\cup \{v\})$ that connects $H_1-H_2$ and $H_2-H_1$. By taking a subpath if necessary, we may assume that $P$ is internally disjoint from $H_1\cup H_2\cup \{v\}$. Let $u_1$ and $u_2$ denote the end-vertices of $P$. Since $G[H_1\cup H_2\cup \{v\}]$ is rigid, $\{u_1,u_2\}$ is weakly globally linked in $G$ by Lemma \ref{wglpath}. As every weakly globally linked pair is connected in $G$, we have $u_1u_2\in E$. It follows that $G[H_1\cup H_2]$ is indeed rigid.

 For every pair $x_1,x_2\in H_1\cup H_2$ with $x_1\neq x_2$, there is an $x_1$-$x_2$-path of length two through $v$. Hence, 
$\{x_1,x_2\}$ is weakly globally linked in $G$ by Lemma \ref{wglpath}, and thus $x_1x_2\in E$. It follows that $H_1\cup H_2$ induces a clique in $G$, a contradiction.
\end{proof}
\end{adjustwidth}

\noindent By Lemma \ref{Epilemma}, if $\pi$ is a uniformly random ordering of $V$, then  $\E(|E_\pi|)\geq d|V|.$
This is a contradiction %
by Theorem \ref{theorem:gluck} and the fact that $G_\pi$ is $\mathcal{R}_d$-independent  by Lemma \ref{tlooselemma}.
\end{proof}

\section{Extending the results}\label{sec:sharp}
In this section, we address two related questions. The first one is how close a $d(d+1)$-connected graph can be to not being (globally) rigid. The second one is how far a $k$-connected graph can be from being rigid when $k<d(d+1)$. The arguments will be similar to those we have seen and thus we will only sketch the proofs.
\newpage
\begin{theorem}
Every $d(d+1)$-connected graph is $\binom{d+1}{2}+1$-redundantly rigid and $\binom{d+1}{2}$-redundantly globally rigid in $\R^d$.
\end{theorem}
\begin{proof}[Proof sketch.]
 Let $d\geq 2$. Suppose, for a contradiction, that $G=(V,E)$ is a non-rigid graph that can be made $d(d+1)$-connected by adding a set $E'$ of at most $\binom{d+1}{2}$ new edges. 
Assume that $G$ has a minimum number of vertices and, with respect to that, a maximum number of edges. Then every linked pair is connected in $G$. Furthermore, for each $v\notin V(E')$, $N_G(v)$ does not induce a clique and $\deg_G(v)\geq d(d+1)$, and hence $\E(\deg_{G_\pi,\piny}(v))\geq d$ follows from the proof of Lemma \ref{Epilemma}. However, for $v\in V(E')$, $N_G(v)$ might induce a clique and we might have $\deg_G(v)< d(d+1)$. Still, for $v\in V(E')$, it follows from Lemma \ref{Epilemma}(1) that $\E(\deg_{G_\pi,\piny}(v))\geq d-\frac{1}{2}\cdot\frac{d(d+1)}{\deg_G(v)+1}$.
By summing over the vertices we obtain 
$$\E(|E_\pi|)\geq d|V| - \sum_{v\in V(E')} \frac{1}{2}\cdot\frac{d(d+1)}{\deg_G(v)+1}\geq d|V|-\sum_{v\in V(E')} \frac{1}{2}\cdot\frac{d(d+1)}{d(d+1)-\deg_{E'}(v)+1}.$$

\noindent If $|E'|=\binom{d+1}{2}$ and the edges of $E'$ are independent (i.e., $|V(E')|=2|E'|=d(d+1)$), then the right-hand side equals $d|V|-\binom{d+1}{2}$. It is easy to prove that if this condition is not met, the right-hand side becomes even larger.  
This contradicts the fact that $G$ is not rigid.

The proof of $\binom{d+1}{2}$-redundant global rigidity is similar. We omit the details. 
\end{proof}

Take two complete graphs on  $d(d+1)$ vertices and connect them with $d(d+1)$ independent edges.
The resulting graph shows that the constants $\binom{d+1}{2}+1$ and $\binom{d+1}{2}$ are best possible.\medskip

For $1\leq k\leq d$ let $m_{d,k}= \frac{k}{2} $, and for $d+1\leq k< d(d+1)$ let $m_{d,k}= d+\frac{1}{2}-\frac{1}{2}\cdot\frac{d(d+1)}{k}$. %

\begin{theorem}\label{minrankkconneted}
For $1\leq k< d(d+1)$, if $G$ is a non-rigid $k$-connected graph in $\R^d$, then $r_d(G)\geq m_{d,k}|V|.$
\end{theorem} 
\begin{proof}[Proof sketch.] Suppose, for a contradiction, that $G=(V,E)$ is a counterexample with a minimum number of vertices and, with respect to that, a maximum number of edges. Suppose that $k\geq d+1$. It can be deduced from $m_{d,k}<\min(k,d)$ that, for $v\in V$, $N_G(v)$ does not induce a clique. 
Thus, the proof of Lemma \ref{Epilemma}  implies $\E(|E_\pi|)\geq m_{d,k}|V|$, which gives a contradiction. (Note the first two lines of Lemma \ref{Epilemma}(3).) 

Suppose that $k\leq d$. It is easy to prove that for $v\in V$
$$\E(\deg_{G_\pi,\piny}(v))\geq\E(\min(\deg_{G,\piny}(v),d))= \sum_{i=0}^{\deg_G(v)}\frac{\min(i,d)}{\deg_G(v)+1}\geq \sum_{i=0}^k\frac{i}{k+1} =\frac{k}{2}.$$
Thus, we get $\E(|E_\pi|)\geq \frac{k}{2}|V|=m_{d,k}|V|$, a contradiction.
\end{proof}
 Note that if $|V|$ is large, then $m_{d,k}|V| < d|V|-\binom{d+1}{2}$. For $2\leq k< d(d+1)$, there exist infinitely many non-rigid $k$-connected graphs, similar to the example in Figure \ref{fig:nonstiff}, for which in Theorem \ref{minrankkconneted} equality holds. Hence, the constant $m_{d,k}$ is best possible. %

\section{Concluding remarks}\label{sec:concluding}
 \subsection{Non-trivial globally rigid subgraphs}

We say that a globally rigid graph in $\R^d$ is a {\it non-trivial} globally rigid graph in $\R^d$ if it has at least $ d+2$ vertices.
For a graph $G$, let ${\rm grn}^*(G)$ denote the largest $d$ such that $G$ has a non-trivial globally rigid subgraph in $\R^d$, 
or 0 if there is no such $d$. It was proved by 
Garamvölgyi and Jordán
   that 
    Theorem \ref{ddplus1 rigid} would imply a lower bound for ${\rm grn}^*(G)$, see \cite[Theorem 5.9]{GJmgr}.
Their  theorem takes the following form if we combine it with Theorem \ref{ddplus1 rigid}.

\begin{theorem}\cite{GJmgr}\label{grn}
For every graph $G=(V,E)$, we have ${\rm grn}^*(G)\geq \left\lfloor\sqrt{\frac{|E|}{6\cdot|V|}} \right\rfloor.$
\end{theorem}
Complete bipartite graphs show that the  lower bound given by Theorem \ref{grn}
is asymptotically tight \cite{GJmgr}.

\subsection{Abstract rigidity matroids and cofactor matroids}
Let $\Mat$ be a matroid defined on the edge set $K$ of a complete graph $(V,K)$, with rank function $r_\Mat$ and closure operator $\langle \cdot\rangle$. 
 We say that $\Mat$ is a \emph{$d$-dimensional abstract rigidity matroid} if (1) and (2) hold.
 
\begin{enumerate}[(1),noitemsep,topsep=0pt]
\item If $E, F \subseteq K$ and $|V(E) \cap V(F)| < d$, then 
$\langle E \cup F\rangle \subseteq K(V(E)) \cup K(V(F))$.
\end{enumerate}

\noindent By putting $F=\emptyset$, it follows from (1) that $\langle E\rangle\subseteq K(V(E))$. An edge set $E\subseteq K$
is called  \emph{$\Mat$-rigid} if $\langle E\rangle = K(V(E))$. A subgraph $(U,E)$ of $(V,K)$ is called {\it $\Mat$-rigid} if $\langle E\rangle = K(U)$.

\begin{enumerate}[(2),noitemsep,topsep=0pt]
\item If $E,F \subseteq K$ are $\Mat$-rigid and $|V(E)\cap V(F)|\geq  d$, then $E\cup F$ is $\Mat$-rigid. 
\end{enumerate}

\noindent Suppose that $\mathcal{M}$ is a $d$-dimensional abstract rigidity matroid. 
 We say that a subgraph $(U,E)$ of $(V,K)$ is \emph{$\Mat$-independent}   if $E$ is $\Mat$-independent.  
A vertex pair $\{u,v\}\subseteq U$ is called \emph{$\Mat$-linked} in $(U,E)$ if $uv\in \langle E\rangle$.\smallskip

 Let $E\subseteq K$. Then the following hold  \mbox{\cite[Lemma 2.5.6]{GSS}}.
\begin{itemize}[label=\raisebox{0.25ex}{\small$\bullet$},noitemsep,topsep=0pt]
  \item  If $|V(E)|\leq d$, then $E$ is $\Mat$-independent.
  \item If %
  $|V(E)|\geq d+1$, then $r_\Mat(E)\leq d|V(E)|-\binom{d+1}{2}$.
  \item If %
  $|V(E)|\geq d+1$ and $E$ is $\Mat$-rigid, then $r_\Mat(E)= d|V(E)|-\binom{d+1}{2}$.
  \item The $d$-dimensional 0-extension operation preserves $\Mat$-independence. %
\end{itemize}
If the $d$-dimensional 1-extension operation  also preserves $\Mat$-independence, then $\Mat$ is called a {\it 1-extendable abstract rigidity matroid} \cite{Ng}. Note that the rigidity matroid of $(V,K)$ is a \mbox{1-extendable} abstract rigidity matroid. 

 Using the properties listed above, the proof of Theorem \ref{ddplus1 rigid} yields the following statement. %

\begin{theorem}\label{arm}
Let %
$(V,K)$ be a complete graph and $\mathcal{M}$ a 1-extendable $d$-dimensional abstract rigidity matroid on $K$. Then every $d(d+1)$-connected subgraph $(U,E)$ of $(V,K)$ is $\mathcal{M}$-rigid.
\end{theorem}

For $d\geq 2$, %
{\it generic $C_{d-1}^{d-2}$-cofactor matroids} of complete graphs are also examples of 1-extendable $d$-dimensional abstract rigidity matroids, see \cite[Theorem 11.3.10, Corollary 11.3.15]{WhitM}. Thus Theorem \ref{arm} implies the following statement, a conjecture of Whiteley \mbox{\cite[Conjecture 11.5.2(2)]{WhitM}}.

\begin{theorem}\label{cofactor}
For $d\geq 2$, every $d(d+1)$-connected graph is generically $C_{d-1}^{d-2}$-rigid.
\end{theorem}
 Theorem \ref{cofactor} is a generalization of a recent result of Clinch, Jackson and Tanigawa, which states that 12-connected graphs are generically $C_{2}^{1}$-rigid \cite{CJT2}.

\section{Acknowledgements}
I am grateful to Tibor Jordán, who introduced me to rigidity theory and the problems discussed in this paper. I thank him for many useful remarks on the manuscript and for suggesting Theorems \ref{arm} and \ref{cofactor}.

\end{document}